\documentclass[12pt]{amsart}

\usepackage[pagebackref]{hyperref} 
\hypersetup{
	colorlinks,
	citecolor=blue,
	filecolor=blue,
	linkcolor=blue,
	urlcolor=blue
}

\usepackage{amsfonts}
\usepackage{amssymb}
\usepackage{amsmath}
\usepackage{amsthm}
\usepackage{verbatim}
\usepackage{esint} 
\usepackage{booktabs}
\usepackage{color}

\newcommand{\R}{{\mathbb R}}
\newcommand{\set}[1]{\left\{#1\right\}}
\newcommand{\dr}{\partial}
\DeclareMathOperator{\divg}{div}

\newcommand{\om}{\Omega}

\def\XXint#1#2#3{{\setbox0=\hbox{$#1{#2#3}{\int}$ }
\vcenter{\hbox{$#2#3$ }}\kern-.6\wd0}}
\def\lbb{\mbox{$\,$---}}
\def\Yint#1{\mathchoice
{\XXint\displaystyle\textstyle{#1}}%
{\XXint\textstyle\scriptstyle{#1}}%
{\XXint\scriptstyle\scriptscriptstyle{#1}}%
{\XXint\scriptscriptstyle\scriptscriptstyle{#1}}%
\!\iint}
\def\fiint{\Yint\lbb}

\normalbaselines
\DeclareMathOperator *{\diam}{diam}
\DeclareMathOperator *{\dist}{dist}

\newtheorem{theorem}{Theorem}[section]

\newtheorem{corollary}[theorem]{Corollary}
\newtheorem{proposition}[theorem]{Proposition}
\newtheorem{definition}{Definition}[section]

\renewcommand{\L}{{L}} 
\newcommand{\Lloc}{\L_{\operatorname{loc}}} 
\newcommand{\W}{{W}}
\newcommand{\Hdot}{\dot{H}\protect{\vphantom{H}}} 
\newcommand{\HT}{H_t} 
\newcommand{\IC}{\mathbb{C}}
\newcommand{\loc}{\operatorname{loc}}
\newcommand{\E}{\mathsf{E}} 
\newcommand{\C}{{C}} 
\newcommand{\dhalf}{D_t^{1/2}} 
\newcommand{\mS}{{\mathcal S}} 
\newcommand{\pd}{\partial}
\renewcommand{\d}{\, \mathrm{d}} 

\numberwithin{equation}{section}

\theoremstyle{plain}

\begin{document}

\title[On the half-time derivative for parabolic PDEs]{On the regularity problem for parabolic operators and the role of half-time derivative}

\author{Martin Dindo\v{s}}
\address{School of Mathematics,
         The University of Edinburgh and Maxwell Institute of Mathematical Sciences, UK}
\email{M.Dindos@ed.ac.uk}

\maketitle

{\centering\footnotesize \it This paper is dedicated to my great friend and collaborator Jill Pipher.\par}

\begin{abstract} In this paper we present the following result on regularity of solutions of the second order parabolic equation $\dr_t u - \divg (A \nabla u)+B\cdot \nabla u=0$ on cylindrical domains of the form 
$\Omega=\mathcal O\times\mathbb R$ where $\mathcal O\subset\mathbb R^n$ is a uniform domain
(it satisfies both interior corkscrew and Harnack chain conditions) and has a boundary that is  $n-1$-Ahlfors regular. Let $u$ be a solution of such PDE
in $\Omega$ and the non-tangential maximal function of its gradient in spatial directions $\tilde{N}(\nabla u)$ belongs to $L^p(\partial\Omega)$ for some $p>1$. Furthermore, assume that for $u|_{\partial\Omega}=f$
we have that $D^{1/2}_tf\in L^p(\partial\Omega)$. Then both $\tilde{N}(D^{1/2}_t u)$ and $\tilde{N}(D^{1/2}_tH_t u)$ also belong to $L^p(\partial\Omega)$, where $D^{1/2}_t$ and $H_t$ are the half-derivative and the Hilbert transform in the time variable, respectively. We expect this result will spur new developments in the study of solvability of the $L^p$ parabolic Regularity problem as thanks to it it is now possible to formulate the parabolic Regularity problem on a large class of time-varying domains.
\end{abstract}

\section{Introduction} The study of the $L^p$ boundary value problems for elliptic and parabolic PDEs has a long and interesting history. Let $\Omega\subset{\mathbb R^n}\times\mathbb R$ be a space-time domain and consider the parabolic differential equation on $\Omega$ of the form
\begin{equation}\label{E:pde}
	\begin{cases}
		\dr_t u - \divg (A \nabla u)+B\cdot \nabla u=0  & \text{in } \Omega, \\
		u   = f                                   & \text{on } \partial \Omega,
	\end{cases}
\end{equation}
where $A= [a_{ij}(X, t)]$ is a $n\times n$ matrix satisfying the uniform ellipticity condition with $X \in \R^n$, $t\in \R$ and $(X,t)\in\Omega$.
That is, there exists positive constants $\lambda$ and $\Lambda=\|A\|_{L^\infty}$ such that
\begin{equation}
	\label{E:elliptic}
	\lambda |\xi|^2 \leq \sum_{i,j} a_{ij}(X,t) \xi_i \xi_j \
\end{equation}
for almost every $(X,t) \in \Omega$ and all $\xi \in \R^n$. We shall assume that $|B(X,t)|\lesssim \delta^{-1}(X,t)$,
where $\delta$ is the parabolic distance of a point $(X,t)$ to the boundary of $\partial\Omega$. 
The elliptic analogue of this problem is that $\Omega\subset{\mathbb R^n}$ and the corresponding PDE is $\divg (A \nabla u)+B\cdot \nabla u=0$ in $\Omega$.

The most classical boundary value problem is the $L^p$ Dirichlet problem ($1<p<\infty$) where for any given boundary data $f\in L^p(\partial\Omega)$ we want to find a solution  $u:\Omega\to\mathbb R$ of \eqref{E:pde} such that its non-tangential maximal function $N(u)$ belongs to $L^p(\partial\Omega)$ and the estimate
\begin{equation}\label{defDP}
\|N(u)\|_{L^p(\partial\Omega)}\le C\|f\|_{L^p(\partial\Omega)}
\end{equation}
holds for some $C>0$ only depending on the PDE, domain and $p$. The estimate \eqref{defDP} makes perfect sense in both elliptic and parabolic settings. The non-tangential maximal function $N$ is defined using non-tangential approach regions $\Gamma(\cdot)$ with vertices on $\partial\Omega$. Very minimal regularity of the boundary domain $\Omega$ is required to meaningfully define these. In the elliptic settings it suffices to have $\Omega$ satisfying the corkscrew condition and $n-1$-Ahlfors regular boundary (in order to define measure on $\partial\Omega$). See \cite{AHMMT, DPext, MPT} and many others.\vglue1mm

There has been a recent substantial progress in considering the parabolic $L^p$ Dirichlet problem on domains with minimal regularity. In \cite{HLN1, HLN2} the notion of parabolic uniform rectifiability was introduced and subsequently \cite{BHHLN} has shown that the bounded solutions of the heat equation $\partial_t-\Delta u=0$ on $\Omega$ have the usual Carleson measure estimates.

To set this into a wider context, in a series of works J.
Lewis and his collaborators \cite{H1, H2, HL96, HL2, LM, LS}, showed that the \lq\lq good"
parabolic graphs for parabolic singular integrals and parabolic potential theory are
regular $Lip(1,1/2)$ graphs, that is, graphs which are $Lip(1,1/2)$ (in space-time coordinates) and which possess extra regularity in time in the sense that a (non-local)
half-order time-derivative of the defining function of the graph belongs to parabolic BMO space. Papers such as \cite{DDH, DH, DPP, HL} show solvability of the parabolic Dirichlet problem on domains of this type under very mild assumptions on coefficients of the parabolic PDE.
Recently,  \cite{BHMN} has shown that  these conditions on the domain are both necessary and sufficient on graph-like domains. The concept of parabolic uniform rectifiability is a generalisation of this concept to non-graph domains.\vglue2mm

The $L^p$ Regularity problem, is again a Dirichlet problem but with more regular boundary data. The elliptic problem has seen a substantial recent development with solvability established on Lipschitz domains under small Carleson condition of coefficients (\cite{DPcpl, DPR, DR}), and large Carleson condition (\cite{DHP, MPT}). The later paper actually considers domains more general than Lipschitz, namely domains that satisfy the corkscrew condition and have uniformly $n-1$-rectifiable boundary.  The formulation of what is the $L^p$ 
Regularity problem is (now) non-controversial on a very general class of domains. The use of the Haj\l{}asz-Sobolev space was a crucial ingredient in the \cite{MPT} paper on the Regularity problem. For the case of coefficients independent of one (transversal) direction see \cite{HKMP2}.

Given a boundary data $f$ having one derivative in $L^p(\partial\Omega)$ we see a solution $u$ such that we have one extra derivative on the left-hand side of \eqref{defDP}, that is
\begin{equation}\label{defDPelipt}
\|\tilde{N}(\nabla u)\|_{L^p(\partial\Omega)}\le C\|\nabla_T f\|_{L^p(\partial\Omega)}.
\end{equation}
Here $\nabla_T f$ is the tangential derivative of $f$ (which is a well defined object on Lipschitz domains); on more general domains we consider instead on the right-hand side the norm in Haj\l{}asz-Sobolev space 
$\dot{M}^{1,p}(\partial\Omega)$ (see below for definition).\vglue2mm

In order to properly formulate the Regularity problem for parabolic PDEs we first note that due to natural parabolic scaling $\nabla u$ scales like a half derivative time-derivative $D^{1/2}_tu $, since the second derivative $\nabla^2 u$ behaves like $\partial_t u$. However, $D^{1/2}_tu $ is a non-local object which requires some thoughts on how to deal with it and what conditions to impose.

The initial formulation of the  Regularity problem is due Fabes and Rivi\`ere \cite{FR} who formulated it for $C^1$ cylinders which was then re-formulated to Lipschitz cylinders by Brown \cite{Bro87, Bro89}. Here by  Lipschitz cylinders we understand domains of the form $\Omega=\mathcal{O}\times [0,T]$ where $\mathcal O\subset\mathbb {R}^n$ is a bounded Lipschitz domain  and time variable runs over a bounded interval $[0,T]$. 
Subsequent authors such as Mitrea \cite{M}, Nystr\"om \cite{Nys06} or 
Castro-Rodr\'iguez-L\'opez-Staubach \cite{CRS} have dealt with this issue in a similar way.

Brown considered solutions  on $\Omega$ with initial data $u=0$ at $t=0$ and naturally extending $u$ by zero for all $t<0$. Then a definition of half-derivative using fractional integrals $I_\sigma$ 
for functions $f\in C^\infty(-\infty,T)$ is used to give meaning to $D^{1/2}_t f$ and $D^{1/2}_t u$. 

A variant of this approach (which we present here) is to consider the parabolic PDE on on infinite cylinder 
$\Omega=\mathcal O\times \mathbb R$. This can be always be achieved by extending the coefficients by setting
$A(X,t)=A(X,T)$ for $t>T$ and $A(X,t)=A(X,0)$ for $t<0$. Since the parabolic PDE has a defined direction of time the solution on the infinite cylinder $\Omega=\mathcal O\times \mathbb R$ with boundary data $u\big|_{\partial\Omega}=0$ for $t<0$ will coincide with the solution on the finite cylinder $\mathcal O\times [0,T]$ with zero initial data.\vglue1mm

Having a well-defined half-derivative the authors \cite{Bro89, CRS, M, Nys06} then define the parabolic Regularity problem by asking for
\begin{equation}\label{defDPpara}
\|\tilde{N}(\nabla u)\|_{L^p(\partial\Omega)}+\|\tilde{N}(D^{1/2}_tu)\|_{L^p(\partial\Omega)}\le C(\|\nabla_T f\|_{L^p(\partial\Omega)}+\|D^{1/2}_t f\|_{L^p(\partial\Omega)}).
\end{equation}
Some authors also ask for control of $\|\tilde{N}(D^{1/2}_tH_tu)\|_{L^p}$, where $H_t$ is the one-dimensional Hilbert transform in the $t$-variable.

This approach is reasonable, if the domain is not time-varying, i.e. of the form $\mathcal O\times (t_0,t_1)$ but 
runs immediately into an obvious issue when this is no longer true, as any reasonable definition of half-derivative requires $u(X,t)$ to be defined for all $t\in\mathbb R$ or at least on a half-line.\vglue1mm

As an illustration, consider the simplest possible case of a domain of the form 
$\om=\set{(x,x_n,t):\, x\in\R^{n-1},\, x_n>\psi(x,t), \, t\in\R}$, for some continuous function $\psi:\mathbb R^{n-1}\times\mathbb R\to\mathbb R$. Unless $\psi(x,t)=\psi(x)$, i.e. the function is time-independent the issue with defining half-derivative will arise (at least for points  near the boundary). 

Recalling works J. Lewis and his collaborators \cite{HL99} mentioned above, under the condition
\begin{equation}\label{psicond}
\|\nabla \psi\|_{L^\infty}<\infty\qquad\mbox{and}\qquad D_t^{1/2}\psi\in BMO(\mathbb R^{n-1}\times \mathbb R).
\end{equation}
a natural change of variables in the parabolic settings (certainly motivated by analogous change of variables in the elliptic settings as in for example \cite{N}) maps 
$\rho:\Omega\to U:=\mathbb R^n_+\times \mathbb R$ which is domain that is constant in time. By transferring the parabolic PDE from $\Omega$ to $U$ we are now in a the previous situation where  half time-derivative can be defined without an issue. The map $\rho$ is bijection that also preserves ellipticity and has comparable non-tangential cones and norms derived from them. However, we pay certain price, namely that the new PDE on $U$ has a drift (first-order term) of the form $B\cdot\nabla u$. For more details see for example \cite{HL}. This is also the reason why our PDE \eqref{E:pde} does include the drift term $B\cdot\nabla u$ so that the new PDE on $U$ falls into the framework considered here. 

\vglue1mm

We therefore ask in this paper the following key question: 
$$\mbox{\it Is the presence of the term $\|\tilde{N}(D^{1/2}_tu)\|_{L^p(\partial\Omega)}$ and perhaps also}$$
$$\mbox{\it 
$\|\tilde{N}(D^{1/2}_tH_tu)\|_{L^p(\partial\Omega)}$ needed in \eqref{defDPpara}?}$$

If not, then asking for control only of $\|\tilde{N}(\nabla u)\|_{L^p(\partial\Omega)}$ opens significant new avenues of research. It might be possible to consider domains $\Omega$ that satisfy
as in \cite{HLN1, HLN2} parabolic uniform rectifiability  but are not necessary graph or locally graph-domain.\vglue2mm

Our main result is that this is indeed the case. That is if $\partial\Omega=\mathcal O\times\mathbb R$
so that we can define the operators $D^{1/2}_t$ and $H_t$ and $u:\Omega\to \mathbb R$ is a solution then bounds for $\|\tilde{N}(D^{1/2}_tu)\|_{L^p(\partial\Omega)}$ and 
$\|\tilde{N}(D^{1/2}_tH_tu)\|_{L^p(\partial\Omega)}$ follow from bounds for just $\|\tilde{N}(\nabla u)\|_{L^p(\partial\Omega)}$.

\begin{theorem}\label{timp} Fix $1<p<\infty$ and consider $\Omega={\mathcal O}\times \mathbb{R}$ such that 
$\mathcal O\subset\mathbb R^n$ is a uniform domain (aka $1$-sided NTA domain) and has 
$n-1$-Ahlfors regular boundary.
Suppose that $u:\Omega\to\mathbb R$ is a function that solves\footnote{$u$ needs to be a reinforced weak solution (defined precisely below) on a set $\Omega$; that is, it satisfies the usual condition of being a weak solution:
$\iint_{\Omega} A\nabla u\cdot{\nabla \phi} \d X \d t - \iint_{\Omega} u \cdot {\partial_{t}\phi} \d X  \d t=0$ for test functions $\phi$ and additionally $u\in H_{\loc}^{1/2}(\R; \L^2_{\loc}(\mathcal O))$.}
\begin{align}\label{eq-pp}
\begin{cases}
Lu&=\partial_tu-\mbox{\rm div}(A\nabla u)+B\cdot\nabla u =0 \quad\text{ in } \Omega\\
u|_{\partial \Omega} &= f \quad\text{ on } \partial \Omega.
\end{cases}
\end{align}
Here $A:\Omega\to M_{n\times n}(\mathbb R)$ is a bounded uniformly elliptic matrix-valued function and $B:\Omega\to {\mathbb R}^n$ is a vector such that $|B|\lesssim\delta(X,t)^{-1}$. The solution $u$ is understood to attain it boundary data $f$ by limiting non-tangentially almost everywhere.\vglue1mm

Then there exist a constant $C(L,\Omega,p)>0$ (but not on $u$) such that if
$\tilde{N}(\nabla u)$ belongs to $L^p(\partial\Omega)$ and $f$ belongs to $\dot{L}_{1,1/2}^p(\partial\Omega)$ we have the estimate:
\begin{align}
\label{Half-main}
\|\tilde{N}(D^{1/2}_tu)\|_{L^p(\partial\Omega)}+\|\tilde{N}(D^{1/2}_tH_tu)\|_{L^p(\partial\Omega)}\le C\left[\|\tilde{N}(\nabla u)\|_{L^p(\partial\Omega)}+\|f\|_{\dot{L}^p_{1,1/2}(\partial\Omega)}\right].
\end{align}
\end{theorem}\vglue2mm

\noindent{\it Remark 1.} Although not stated here this result also applies to parabolic systems. There are only  two places in our argument where the PDE is actually used and for any parabolic system  
the argument given here also applies.\vglue2mm

\noindent{\it Remark 2.} In the definition above we are very liberal about the notion of a solution. Clearly, assuming some $L^p$ norm of $\tilde{N}(\nabla u)$ implies that in the interior of $\Omega$ we have a well-defined gradient $\nabla u\in L^2_{loc}(\Omega)$. We further need that half-derivative $D^{1/2}_t$ can be moved from $\partial_t$ onto a test function as otherwise it is impossible to make sense of $D^{1/2}_tu$ and discuss nontangential maximal function of it.
These solutions (called reinforced weak solutions) can be constructed using variety of techniques, for example 
via the \lq\lq hidden coercivity method" as in \cite{ABES, AEN} or by the usual construction of the parabolic measure. The class $\dot {\E}_{\loc}(\Omega)$ is the minimal one in which the objects such as $\tilde{N}(D^{1/2}_tu)$, $\tilde{N}(D^{1/2}_t\HT u)$ are well-defined.
\vglue2mm

In the context of Theorem \ref{timp} let us precisely define the space $\dot{L}^p_{1,1/2}(\partial\Omega)$ we consider here. Recall that 
with $\mathcal O$ as in Theorem \ref{timp}, we may consider the Haj\l{}asz-Sobolev space $\dot M^{1,p}(\partial O)$ (as in \cite{MPT}) consisting of functions $f\in L^p_{loc}(\partial \mathcal O)$ such that for some  $g\in L^p(\partial \mathcal O)$ 
\begin{equation}\label{split1bz}
|f(x)-f(y)|\le |x-y|[g(x)+g(y)]\qquad\mbox {for almost every }x,y\in\partial\mathcal O.
\end{equation}
The norm is given as the infimum of $L^p$ norms of all such $g$. This motivates the following
natural generalisation of the usual parabolic $\dot{L}^p_{1,1/2}$ space on $\partial\mathcal O\times\R$.

\begin{definition} Let $\mathcal O\subset\mathbb R^n$ is a uniform domain with  $n-1$-Ahlfors regular boundary. Given a space-time function $f:\partial\mathcal O\times\mathbb R
\to\mathbb R$ we define $\dot M_x^{1,p}(\partial \mathcal O\times\mathbb R)$ as follows. Denote by $\mathcal D_s(f)$ (the subscript $s$ indicates we are only considering the spatial \lq\lq gradient") the space consisting of function $g\in L_{loc}(\partial\mathcal O\times\mathbb R)$ such that for  a.e. $t\in\mathbb R$ we have:
\begin{equation}\label{m1}
|f(x,t)-f(y,t)|\le |x-y|[g(x,t)+g(y,t)]\qquad\mbox {for almost every }x,y\in\partial\mathcal O,
\end{equation}
with $\|f\|_{\dot M_x^{1,p}}:=\inf\{\|g\|_{L^p}:\, g\in\mathcal D_s(f)\cap L^p\}$.
We shall then define
\begin{equation}\label{SpaceL}
\dot{L}^p_{1,1/2}(\partial\mathcal O\times\mathbb R)=\{f\in L^p_{loc}(\partial\mathcal O\times\mathbb R):\,\|f\|_{\dot M_x^{1,p}}+\|D^{1/2}_tf\|_{L^p}<\infty\}.
\end{equation}
\end{definition}

Finally, we may now extend the definition given above to the class of time-varying domains and define the parabolic  $L^p$  Regularity problem for such class of domains. 

\begin{definition} \label{DefRpar} Let $\Omega\subset \mathbb R^n\times \mathbb R$ be space-time domain
such that for a.e. $t\in\R$ the sets $\mathcal O_t:=\Omega\cap \{t=const\}$ are uniform domains with $n-1$-Ahlfors regular boundary $\partial\mathcal O_t$ and the constants in the Definitions \ref{def1.cork}-\ref{def-ahlf} do not depend on $t$.

It follows that for any fixed $1<p<\infty$ we have for $f:\partial\Omega\to \mathbb R$ a well-defined notion of one spatial derivative, namely by asking for $g\in L^p(\partial\Omega)$
\begin{equation}
|f(X,t)-f(Y,t)|\le |X-Y|(g(X,t)+g(Y,t)),
\end{equation}
for a.e. $t\in \mathbb R$ and all $(X,t)$ and $(Y,t)$ on $\partial\Omega$ except perhaps a set of surface measure zero. Taking infimum of $L^p$ norms over all such functions $g$ defines the $\dot M_x^{1,p}(\partial\Omega)$ norm.

Similarly, assume a well-defined notion of fractional half-regularity of $f$ in the $t$-variable and call the resulting space of functions $\dot{\mathcal X}^p_{1,1/2}$. When $\Omega$ is a time cylinder the space 
$\dot{\mathcal X}^p_{1,1/2}$ has to coincide with the space  $\dot{L}^p_{1,1/2}(\partial\Omega)$ defined earlier.\vglue1mm

We say that the $L^p$ Regularity problem for the parabolic PDE problem \eqref{eq-pp} is solvable if there exists a linear functional $T:f\mapsto u$ such that $u:\Omega\to\mathbb R$ solves the equation
\eqref{eq-pp},  attains on $\partial \Omega$ datum $f$ non-tangentially almost eveywhere and for some $C=C(L,\Omega,p)>0$ we have:
\begin{align}
\label{RPpar} \quad\|\widetilde{N}(\nabla u)\|_{L^p(\partial
\Omega)}\le C\| f\|_{\dot{\mathcal X}_{1,1/2}^{p}(\partial\Omega)}.
\end{align}
\end{definition}

\begin{corollary} By Theorem \ref{timp}
the above notion of $L^p$ solvability of the Regularity problem coincides with the definition of the Regularity problem as given previously in \cite{Bro87, Bro89, CRS, M, Nys06} on Lipschitz cylinders $\Omega={\mathcal O}\times \mathbb R$ or on
time-varying domains $\om=\set{(x,x_n,t):\, x\in\R^{n-1},\, x_n>\psi(x,t), \, t\in\R}$ for $\psi$ satisfying the usual 
condition \eqref{psicond}.
\end{corollary}

\noindent{\bf Acknowledgement:} The author is grateful to the referees for suggestions that have improved the result and have shortened some of the arguments.

\section{Definitions}
\setcounter{equation}{0}

\subsection{Domains $\Omega=\mathcal O\times\R$.}%
\label{S:paraSobolev}

Let us recall some important definitions. 

\begin{definition}[\bf Corkscrew condition]\label{def1.cork}
 \cite{JK}.  A domain $\mathcal O\subset {\mathbb R}^{n}$
satisfies the {\it Corkscrew condition} if for some uniform constant $c>0$ and
for every surface ball $\Delta:=\Delta(q,r),$ with $q\in \partial\mathcal O$ and
$0<r<\diam(\partial\mathcal O)$, there is a ball
$B(X_\Delta,cr)\subset B(q,r)\cap\mathcal O$.  The point $X_\Delta\subset \mathcal O$ is called
a {\it corkscrew point relative to} $\Delta,$ (or, relative to $B$). We note that  we may allow
$r<C\diam(\partial\mathcal O)$ for any fixed $C$, simply by adjusting the constant $c$.
\end{definition}

\begin{definition}[\bf Harnack Chain condition]\label{def1.hc}
\cite{JK}. Let $\delta(X)$ denote the distance of $X \in \mathcal O$ to $\partial \mathcal O$. A domain
$\mathcal O$ satisfies the {\it Harnack Chain condition} if there is a uniform constant $C$ such that
for every $\rho >0,\, \Lambda\geq 1$, and every pair of points
$X,X' \in \mathcal O$ with $\delta(X),\,\delta(X') \geq\rho$ and $|X-X'|<\Lambda\,\rho$, there is a chain of
open balls
$B_1,\dots,B_N \subset \mathcal O$, $N\leq C(\Lambda)$,
with $X\in B_1,\, X'\in B_N,$ $B_k\cap B_{k+1}\neq \emptyset$
and $C^{-1}\diam (B_k) \leq \dist (B_k,\partial\mathcal O)\leq C\diam (B_k).$  The chain of balls is called
a {\it Harnack Chain}.
\end{definition}

\begin{definition}[\bf Uniform domains]\label{def1.1nta}
If $\mathcal O$ satisfies both the Corkscrew and Harnack Chain conditions, then
$\mathcal O$ is a {\it uniform domain}.
\end{definition}

\begin{definition}[\bf Ahlfors regularity]\label{def-ahlf} For a domain $\mathcal O\subset\R^n$ we say that
$\partial \mathcal O$ satisfies $n-1$-dimensional Ahlfors condition (or is $n-1$-Ahlfors regular) if there exists $C>0$ such that
for all $x\in\partial\Omega$ and all $r\le\mbox{diam}(\mathcal O)$
$$C^{-1}r^{n-1}\le\mathcal H^{n-1}(\partial\mathcal O\cap B(x,r))\le Cr^{n-1}.$$
\end{definition}

\subsection{Reinforced weak solutions}

We recall the paper \cite{AEN} that neatly presents the concept of reinforced weak solutions for the parabolic problem we are interested in.\vglue1mm

If $\mathcal O$ is an open subset of $\mathbb R^{n} $, we let $H^1(\mathcal O)=W^{1,2}(\mathcal O)$ be the standard Sobolev space of real valued functions $v$ defined on $\mathcal O$, such that $v$ and $\nabla v$ are in $\L^{2}(\mathcal O;\R)$ and $\L^{2}(\mathcal O;\R^n)$, respectively. A subscripted `$\loc$' will indicate that these conditions hold locally.

We shall say that $u$ is a \emph{reinforced weak solution} of $\partial_t u-\divg(A\nabla u)=0$ on $\Omega=\mathcal O\times \R$ if
\begin{align*}
 u\in \dot {\E}_{\loc}(\Omega):= H_{\loc}^{1/2}(\R; L^2_{\loc}(\mathcal O)) \cap \Lloc^2(\R; W^{1,2}_{\loc}(\mathcal O))
\end{align*}
and if for all $\phi,\psi \in \C_0^\infty(\Omega)$,
\begin{equation}\label{2.10}
\iint_{\Omega}\left[
 A\nabla u\cdot{\nabla (\phi\psi)}+ \HT\dhalf (u\psi)\cdot {\dhalf \phi}+ \HT\dhalf (u\phi)\cdot {\dhalf \psi}\right]\, \d X \d t=0.
 \end{equation}
Here, $\dhalf$ is the half-order derivative and $\HT$ the Hilbert transform with respect to the $t$ variable, designed in such a way that $\partial_{t}= \dhalf \HT \dhalf$. The space $\Hdot^{1/2}(\R)$ is the homogeneous Sobolev space of order 1/2 (it is the completion of $\C_0^\infty(\R)$ for the norm $\|\dhalf (\cdot)\|_{2}$ and, modulo constants, it embeds into the space $\mS'(\R)/\IC$ of tempered distributions modulo constants).
By  $H_{\loc}^{1/2}(\R)$ we mean functions $u$ such that $u\phi\in \Hdot^{1/2}(\R)$ for all $\phi\in \C_0^\infty(\Omega)$.
Here we have departed slightly from \cite{AEN} as in their definition of $\dot {\E}_{\loc}(\Omega)$ the space $H_{\loc}^{1/2}$ is replaced by $ \Hdot^{1/2}$ and $\psi\equiv 1$. For such $u$ \eqref{2.10} simplifies to 
\begin{equation}\nonumber
\iint_{\Omega}\left[
 A\nabla u\cdot{\nabla \phi}+ \HT\dhalf u\cdot {\dhalf \phi}\right]\, \d X \d t=0.
 \end{equation}

 Clearly any $u\in \Hdot^{1/2}(\R)$ also belongs to $H_{\loc}^{1/2}(\mathbb R)$ and hence our notion of reinforced weak solution somewhat weaker than in \cite{AEN}. This can be seen by taking a sequence of functions $\psi_n\in C_0^\infty(\Omega)$ for which $\psi_n\to 1$ as $n\to\infty$.
 
 Our definition has advantage that by taking a cut-off of the function $u$  we might potentially improve the decay of $D^{1/2}_tu$ at infinity. In particular, this matters when considering $u\big|_{\partial\Omega}\in L^p_{1,1/2}(\partial\Omega)$ for $p>2$ as such $u$ might not decay fast enough to have
$u\in \Hdot^{1/2}(\R; \L^2_{\loc}(\mathcal O))$.

At this point we remark that for any $u\in \Hdot^{1/2}(\R)$ and $\phi,\psi\in \C_0^\infty(\R)$ the formula
\begin{align*}
 \int_{\R} \left[\HT\dhalf (u\psi)\cdot {\dhalf\phi}+\HT\dhalf (u\phi)\cdot {\dhalf\psi}\right]\d t = - \int_{\R} u \cdot {\partial_{t}(\phi\psi)} \d t
\end{align*}
holds, where on the right-hand side we use the duality form between $\Hdot^{1/2}(\R)$ and its dual $\Hdot^{-1/2}(\R)$ extending the complex inner product of $\L^2(\R)$. It follows that a reinforced weak solution is a weak solution in the usual sense on $\Omega$ as it satisfies
$u\in \Lloc^2(\R; \W^{1,2}_{\loc}(\mathcal O))$ and for all $\phi\in \C_0^\infty(\Omega)$,
\begin{align*}
 \iint_{\Omega} A\nabla u\cdot{\nabla \phi} \d X \d t - \iint_{\Omega} u \cdot {\partial_{t}\phi} \d X  \d t=0.
\end{align*}
We get this by taking $\psi=1$ on the set where $\phi$ is supported. This implies $\pd_{t}u\in \Lloc^2(\R; \W^{-1,2}_{\loc}(\mathcal O))$. Conversely, as we have already stated, any  weak solution $u$ in    $H_{\loc}^{1/2}(\R; \L^2_{\loc}(\mathcal O))$ is a reinforced weak solution.\vglue2mm

\subsection{Non-tangential approach regions and Non-tangential maximal function}
\noindent{ }

Our notion of solvability of the Dirichlet problem requires  few more definitions.
In the first place, we need to introduce a slightly non-standard notion of a nontangential approach region - such regions are typically referred to as
``cones" when the domain is Lipschitz regular, and ``corkscrew regions" for more general domains that satisfy the corkscrew condition given above. We first define non-tangential approach region on the set $\mathcal O$.\vglue1mm

In the following, the parameter
$a$ is positive and will be referred to as the ``aperture". 
A standard corkscrew region associated with a boundary point $q\in\partial\mathcal O$ is defined (\cite{JK})  to be 
$$\gamma_a(q)=\{X \in \mathcal O: |X-q|< (1+a)\delta(X)\},$$
for some $a>0$ and nontangential maximal functions, square functions are defined in the literature with respect to these regions. Here $\delta(X)=dist(X,\partial\mathcal O)$ is the Euclidean distance of a point $X$ to the boundary. Here and further below we implement the rule that lower case letters (such as $q$ denote points on $\partial\mathcal O$), while upper case letters (such as $X,Y$) denote points inside $\mathcal O$.

We modify this definition in order to achieve a certain geometric property which may not hold 
for the $\gamma_a(q)$ in general. If the domain $\mathcal O$ is Lipschitz we may stop here.

\begin{definition}\label{Cones}
For $Y \in \mathcal O$, let $S_a(Y) := \{q \in \partial \mathcal O: Y \in \gamma_a(q)\}.$
Set 
$$\tilde{S}_a(Y) := \bigcup_{q \in S_a(Y)} \{q'\in\partial\mathcal O:|q-q'|<a\delta(Y)\}.$$
Define 
$$\widetilde{\gamma}_a(q) := \{ Y \in \Omega: q \in \tilde{S}_a(Y)\}.$$
\end{definition}
It was show in \cite{DPext} that these \lq\lq novel" corkscrew regions have the property that 
for any $q \in \partial \mathcal O$,  $\gamma_a(q) \subset \widetilde\gamma_a(q)\subset\gamma_{2a}(q)$.
Thus our $\widetilde\gamma_a(q)$ is sandwiched in between two standard corkscrew regions and is thus itself a 
corkscrew region.\vglue1mm

We take advantage of the product structure of our domain $\Omega$ and for any point $(q,\tau)\in\partial\Omega=\partial\mathcal O\times\mathbb R$ and define the non-tangential cones $\Gamma_a(q,\tau)\subset\Omega$ by simply taking 
$$\Gamma_a(q,\tau)=\{(X,t)\in\Omega: X\in\widetilde\gamma_a(q)\mbox{ and }|t-\tau|<\delta(X)^2\}.$$
This is equivalent to more standard non-tangential parabolic cones defined as
$$\widetilde\Gamma_a(q,\tau)=\{(X,t)\in\Omega: d((X,t),(q,\tau))<(1+a)\delta(X,t)\},$$
where $d((X,t),(q,\tau))$ is the parabolic distance function
$$d((X,t),(q,\tau))=\left(|X-q|^2 + |t-\tau|\right)^{1/2},$$
and
$\delta(X,t)$ is the parabolic distance to the boundary
\[\delta (X,t) = \inf_{(q,\tau) \in \partial\Omega} d((X,t),(q,\tau)). \]  
In our case thanks to the fact the domain as in infinite cylinder we always have $\delta (X,t)=\delta(X)$.
It's a simple exercise to show that for any $a>0$ there exists $b=b(a)>0$ such that 
$$\widetilde\Gamma_{1/b}(q,\tau)\subset \Gamma_a(q,\tau)\subset \widetilde\Gamma_{b}(q,\tau),$$
for all boundary points $(q,\tau)$. Typically we suppress the subscript $a$ and consider it fixed and only write $\Gamma(\cdot)$. The $L^p$ norms of the non-tangential maximal function defined below are comparable for different values of apertures $a$. 

In what follows we denote by $B_r(X,t)$ a parabolic ball that is a collection of points in the background space $\R^n\times\R$ of parabolic distance $d$ less than $r$ to the point $(X,t)$.

\begin{definition}\label{D:NT} 
For $\Omega$ as above, 
the non-tangential maximal function $\tilde{N}_{p,a}$ or just $\tilde{N}_{p}$ is defined using $L^p$ averages over the interior parabolic balls in the domain $\Omega$. 
Specifically, given $w\in L^p_{loc}(\Omega)$ we set for $(q,\tau)\in\partial\Omega$:
\begin{equation}\label{SSS-3}
\tilde{N}_{p,a}(w)(q,\tau):=\sup_{(X,t)\in\Gamma_{a}(q,\tau)}W_p(w)(X,t),
\end{equation}
where, at each $(X,t)\in\Omega$, 
\begin{equation}\label{w}
W_p(w)(X,t):=\left(\fiint_{B_{\delta(X,t)/2}(X,t)}|w(Y,s)|^{p}\,dY\,ds\right)^{1/p}.
\end{equation}
When we omit the subscript $p$ we always understand that $p=2$, hence $\tilde{N}=\tilde{N}_{2}=\tilde{N}_{2,a}$ for some fixed $a>0$. If we consider a different value than $p=2$ it will be explicitly stated. We note that the radius of the parabolic ball $\delta(X,t)/2$ in \eqref{w} is chosen for convenience only, choosing for example $\delta(X,t)/8$ will lead to a different non-tangential maximal function but with comparable $L^p$ norm to the original one which can be seen by considering the level sets, for example.
\end{definition}

The regions $\Gamma_a(p,\tau)$ have the following property inherited from $\gamma_{2a}(p)$: for any pair of points $(X,t), (X',t')$ in
$\Gamma_a(p,\tau)$, there is a Harnack chain of balls connecting them - see Definition \ref{def1.hc}. The centers of the balls
in this Harnack chain will be contained in a corkscrew region $\Gamma_{a'}(p,\tau)$ of slightly larger aperture, where $a'$ depends only the geometric constants in the definition of the domain $\mathcal O$.

\section{Proof of Theorem \ref{timp}}
\setcounter{equation}{0}

\begin{proof}  We postpone the proof of the fully general case and first look at $\Omega=\mathbb R^n_+\times\mathbb R$. Clearly, when $\Omega$ is a Lipschitz domain the matters can be reduced to this case by considering  local consider projections of the $\mathcal O\times\mathbb R$ which will gives us control over truncated version of the non-tangential maximal function. This is the heart of the matter anyway as  considering regularity away from the boundary is easier.

We will address the case when $\mathcal O$ is uniform domain with $n-1$-Ahlfors regular boundary at the end by highlighting the key differences this more general case requires.\vglue2mm

Fix a boundary point $(p,\tau)\in \mathbb R^n$ and the nontangential cone $\Gamma(p,\tau)$ emanating from it. Pick an arbitrary $(X,t)\in \Gamma(p,\tau)\subset \mathbb R^n_+\times\mathbb R=\Omega$. 
Recall that $$\widetilde{N}(D^{1/2}_tu)(p,\tau)=\sup_{(X,t)\in\Gamma(p,\tau)} w(X,t),\mbox{ where: }
w(X,t)=\left(\fiint_{Q(X,t)} |D^{1/2}_t u(Y,s)|^2dY\,ds\right)^{1/2}.
$$
Here we denote by $Q(X,t)$ the parabolic region $Q(X,t)=\{(Y,s):|X-Y|<\delta(X,t)/4\,\mbox{ and }\,|s-t|<\delta^2(X,t)/16\}$.
We note that the $\delta(X,t)=x_n$ is the distance of the point to the boundary. Here the shape of the region surrounding the point $(X,t)$ does not matter as long as it contains and is contained in parabolic balls centered at $(X,t)$ of diameter $\approx \delta(X,t)$. The choice of $\delta(X,t)/4$ and $\delta(X,t)^2/16$ is just for convenience so that an enlargement $3Q(X,t)\subset \subset\Omega$. 

We find a smooth cutoff function $\varphi\in C_0^\infty (\mathbb R)$ such that $\varphi(t')=1$ for $|t-t'|<x_n^2/8$ and $\varphi(t')=0$ for $|t-t'|>3x_n^2/8$. Clearly we may also assume that $\|\partial_t\varphi\|_{L^\infty}\lesssim x_n^{-2}$. Consider the projection $\pi:\Omega\to \partial\Omega$ defined via
$$\pi:(X,t)=(x,x_n,t)\mapsto (x,0,t).$$

Let us also denote by $\Delta$ the projection of $Q$ onto the boundary, that is $\Delta=\Delta(X,t)=\{(y,0,s): \exists (y,y_n,s)\in Q(X,t)\}=\pi(Q(X,t))$. Sets like $2\Delta(X,t)$, $3\Delta(X,t)$
are defined analogously.
We write 
\begin{equation}\label{split1}
D^{1/2}_tu= D^{1/2}_t((u-\textstyle\fiint_{Q(X,t)}u)\varphi) +D^{1/2}_t((u-g)(1-\varphi))
\end{equation}
\begin{equation}\qquad\qquad\nonumber
+D^{1/2}_t((g-\textstyle\fint_{\Delta} f)(1-\varphi))+D^{1/2}_t((\textstyle\fiint_{Q(X,t)}u-\textstyle\fint_{\Delta} f)\varphi).
\end{equation}
Here we used a shortcut by denoting $f=u\circ\pi$ which makes sense since $u\big|_{\partial\Omega}=f$.
We denoted by $g$ the moving average of $u$ over regions $Q(X,t)$, that is for $r=\delta(X,t)/4$
\begin{equation}\label{defg}
g(s)=\fiint_{Q_r(X,s)}u,
\end{equation}
and in particular $\textstyle\fiint_{Q(X,t)}u=g(t)$.

With the decomposition \eqref{split1} we have an estimate of $w(X,t)$ defined above by
\begin{equation}\label{split2}
w(X,t)\lesssim w_1(X,t)+w_2(X,t)+w_3(X,t)+w_4(X,t),
\end{equation}
where $w_i$ is defined as the $L^2$ average of the $i$-th term of the right-hand side of \eqref{split1}. 

We start with the term $w_1$. We use the fact  that $u$ is a solution of \eqref{eq-pp}. We only need it to estimate the local term defining $w_1$. By \eqref{split1} we want to consider averages of  $D^{1/2}_t((u-\textstyle\fiint_{Q(X,t)}u)\varphi)$.

Recall that currently $\varphi$ is a cutoff function only in the $t$ variable. Consider another function $\psi\in C_0^\infty(\mathbb R^n_+)$ such that $|\nabla \psi|\lesssim x_n^{-1}$, $0\le \psi\le 1$ and the product function $\Phi(X,t)=\psi(X)\varphi(t)$ equals to $1$ on $2Q(X,t)$ and vanishes outside $3Q(X,t)$. We clearly have
$$\iint_{Q(X,t)}|D^{1/2}_t((u-\textstyle\fiint_{Q(X,t)}u)\varphi)|^{2}=\displaystyle\iint_{Q(X,t)}|D^{1/2}_t(( u-\textstyle\fiint_{Q(X,t)}u)\Phi)|^{2}$$
\begin{equation}\label{split49}
\lesssim \iint_{\mathbb R^n_+\times\mathbb R}|D^{1/2}_t(( u-\textstyle\fiint_{Q(X,t)}u)\Phi)|^{2}=
\displaystyle\iint_{\mathbb R^n_+\times\mathbb R}|D^{1/2}_t(v\Phi)|^{2}.
\end{equation}
Here $v=u-\textstyle\fiint_{Q(X,t)}u$. Hence $v$ solves the same PDE as $u$, namely that
$$Lv=\partial_tv-\mbox{\rm div}(A\nabla v)+B\cdot\nabla v =0.$$
We multiply above PDE by $-H_t(v\Phi)\Phi$ and integrate over our domain, where $H_t$ is the Hilbert transform in the time variable. Consider first the term that contains $v_t$. We will have
\begin{equation}\label{split50}
-\iint_{\mathbb R^n_+\times\mathbb R}(\partial_t v)H_t(v\Phi)\Phi=-
\iint_{\mathbb R^n_+\times\mathbb R}(\partial_t (v\Phi))H_t(v\Phi)+\iint_{\mathbb R^n_+\times\mathbb R}vH_t(v\Phi)\partial_t\Phi.
\end{equation}
The second term of \eqref{split50} we will treat as an \lq\lq error term" and we bound it. Recall that $H_t$ is an isometry on $L^2(\mathbb R)$ and hence 
$$\int_{t\in\mathbb R}|H_t(v\Phi)|^2\,dt=\int_{t\in\mathbb R}|(v\Phi)|^2\,dt\le \int_{t\in 3Q(X,t)}|v|^2\,dt.$$
Hence the second term of \eqref{split50} has a bound (by Cauchy Schwarz and taking into account the support of $\Phi$):
\begin{equation}\label{split51}
|E|=\left|\iint_{\mathbb R^n_+\times\mathbb R}vH_t(v\Phi)\partial_t\Phi\right|\le x_n^{-2}\iint_{3Q(X,t)}|v|^2.
\end{equation}
Here we have used the fact that $|\partial_t\Phi|\lesssim x_n^{-2}$.
Continuing with the first term on the right-hand side of \eqref{split50} we write $\partial_t=D^{1/2}_tH_tD^{1/2}_t$
and move the half derivative $D^{1/2}_t$ to the other term. This gives us
\begin{equation}\label{splitss}
-\iint_{\mathbb R^n_+\times\mathbb R}(\partial_t (v\Phi))H_t(v\Phi)=
\displaystyle\iint_{\mathbb R^n_+\times\mathbb R}|D^{1/2}_tH_t(v\Phi)|^{2}=\iint_{\mathbb R^n_+\times\mathbb R}|D^{1/2}_t(v\Phi)|^{2}.
\end{equation}
This is precisely the term we aim to estimate (see \eqref{split49}). We now consider remaining terms of our integrated PDE for $v$. We have
$$\iint_{\mathbb R^n_+\times\mathbb R}|D^{1/2}_t(v\Phi)|^{2}=-\iint_{\mathbb R^n_+\times\mathbb R}
\mbox{\rm div}(A\nabla v)H_t(v\Phi)\Phi+\iint_{\mathbb R^n_+\times\mathbb R}
(B\cdot\nabla v)H_t(v\Phi)\Phi+E.
$$
The second term after integrating by parts (and remembering that $\nabla$ and $H_t$ commute) yields three new terms to consider:
$$\iint_{\mathbb R^n_+\times\mathbb R}
(A\nabla v)H_t((\nabla v)\Phi)\Phi+\iint_{\mathbb R^n_+\times\mathbb R}
(A\nabla v)H_t(v\nabla(\Phi))\Phi+\iint_{\mathbb R^n_+\times\mathbb R}
(A\nabla v)H_t(v\Phi)\nabla\Phi.
$$
By Cauchy-Schwarz (applying it to all four terms and using the bounds on $B$) we further obtain bounds by
$$\lesssim \iint_{3Q(X,t)}|\nabla v|^2+\iint_{\mathbb R^n_+\times\mathbb R}|H_t((\nabla v)\Phi)|^2+
\iint_{\mathbb R^n_+\times\mathbb R}|H_t(v\nabla(\Phi))|^2+{x_n}^{-2}\iint_{\mathbb R^n_+\times\mathbb R} |H_t(v\Phi)|^2.$$
Here in the last term we use $|\nabla\Phi|\lesssim x_n^{-1}$. We again use the fact that $H_t$ is an $L^2$ isometry in time with allows us to remove $H_t$ from each term. Thus after taking into account the set on which $\Phi$ is supported the line above simplifies to
$$\lesssim \iint_{3Q(X,t)}|\nabla v|^2+x_n^{-2}
\iint_{3Q(X,t)}|v|^2.$$
Hence after putting all terms together:
$$\iint_{\mathbb R^n_+\times\mathbb R}|D^{1/2}_t(v\Phi)|^{2}\lesssim  \iint_{3Q(X,t)}|\nabla v|^2+x_n^{-2}
\iint_{3Q(X,t)}|v|^2.
$$
Thus we have for $w_1(X,t)$ (since $\nabla v=\nabla u$):
\begin{equation}\label{splitw1a}
w_1(X,t)\lesssim M(\tilde{N}(\nabla u))(P,\tau)+x_n^{-1}\left(\fiint_{3Q(X,t)}|v|^2\right)^{1/2}.
\end{equation}
Here we have  used the fact that $(X,t)\in\Gamma(P,\tau)$ and the definition of $\tilde{N}_2(\nabla u)$. 
For a fixed time $t_0$ let us consider averages of $u$ on a fixed time-slice $t_0$, that is
\begin{equation}\label{uav}
u_{av}(t_0):=\fiint_{Q(X,t)\cap \{t_0=const\}}u(Y,t_0)dY.
\end{equation}
By Poincar\'e inequality in the spatial variables we have that
$$\iint_{3Q(X,t)\cap \{t_0=const\}}|u-u_{av}|^2(Y,t_0)\lesssim x_n^2\iint_{3Q(X,t)\cap \{t_0=const\}}|\nabla u|^2.$$
Thus by \eqref{splitw1a} we see that
\begin{equation}\label{splitw1b}
w_1(X,t)\lesssim M(\tilde{N}(\nabla u))(P,\tau)+x_n^{-1}\left(\fint_{|s-t|<9x_n^2}|u_{av}(s)-\textstyle\fiint_{Q(X,t)}u|^2ds\right)^{1/2}.
\end{equation}
Observe that 
\begin{equation}\label{eq311}
\fiint_{Q(X,t)}u=\fint_{|s-t|<x_n^2} u_{av}(s)\,ds,
\end{equation}
and hence we really need to understand the difference $|u_{av}(s)-u_{av}(s')|$ for different times $s,s'$ in the interval $(t-9x_n^2,t+9x_n^2)$. This has been done in very general settings (quasi-linear parabolic PDEs) in \cite{KL}. In particular by Lemma 3.1 of this paper it follows that 
$$\sup_{s,s'\in (t-9x_n^2,t+9x_n^2)}\left|u_{av}(s)-u_{av}(s')\right|\lesssim x_n\fiint_{3Q(X,t)}|\nabla u|dY\le x_n
\left(\fiint_{3Q(X,t)}|\nabla u|^2\,dY\,dt\right)^{1/2}.
$$
for all $s,\,s'$ with distance to $t$ $\le 9x_n^2$.
This combined with \eqref{eq311} yields
\begin{equation}\label{splitwww}
\left(\fint_{|s-t|<9x_n^2}|u_{av}(s)-\textstyle\fiint_{Q(X,t)}u|^2ds\right)^{1/2}\lesssim  x_n
\left(\fiint_{3Q(X,t)}|\nabla u|^2\,dY\right)^{1/2},
\end{equation}
which is precisely what is needed for \eqref{splitw1b}. Thus, we have
\begin{equation}\label{splitw1}
w_1(X,t)\lesssim M(\tilde{N}(\nabla u))(P,\tau).
\end{equation}

We consider the term $w_2$ next.  For $(Y,s)\in Q(X,t)$  we have that
$$D^{1/2}_t((u- u_{av})(1-\varphi))(Y,s)=-c\int_{s'\in\mathbb R}\frac{(u-u_{av})(Y,\tilde{s})(1-\varphi(Y,\tilde{s})))}{|s-\tilde{s}|^{3/2}}d\tilde{s}.$$
Recall that the support of $1-\varphi$ is outside $2Q$ and hence $|s-\tilde{s}|\ge x_{n}^{2}\approx y_n^2$ on support $1-\varphi$. Thus we don't need to worry about the singularity of this integral at zero. It follows:
$$|D^{1/2}_t((u- u_{av})(1-\varphi))|(Y,s)\lesssim \int_{|s-\tilde{s}|\ge x_n^{2}}\frac{|u-u_{av}|(Y,\tilde{s})}{|s-\tilde{s}|^{3/2}}d\tilde{s}.$$
In particular also for any $\tau\in [t-s-x_n^2/16,t-s+x_n^2/16]=:I_{s}$
$$|D^{1/2}_t((u- u_{av})(1-\varphi))|(Y,s+\tau)\lesssim \int_{|s-\tilde{s}|\ge x_n^{2}}\frac{|u-u_{av}|(Y,\tilde{s}+\tau)}{|s-\tilde{s}|^{3/2}}d\tilde{s}.$$
We now calculate the $L^2$ norm over the region where $(Y,\tau)\in \{|Y-X|<\delta(X)/4\}\times I_s$.
By Minkowski's inequality we have that 
$$\left(\fiint_{Q(X,t)} [D^{1/2}_t((u- u_{av})(1-\varphi))(X',t')]^2\,dt'\,dX'\right)^{1/2}$$
$$\lesssim
\int_{|t'|\ge x_n^2}t'^{-3/2}\left(\fiint_{Q(X,t+t')}|u-u_{av}|^2\right)^{1/2}dt'\le  \sum_{i=0}^\infty 2^{-i/2}\fint_{|t'|\approx 2^i x_n^{2}} x_n^{-1}W(u-u_{av})(X,t+t')d{t'}.$$
Here $W(w)$ denotes the $L^2$ average of a function $w$ over a Whitney ball $B_{\delta(X,t)/2}(X,t)$.

Let us recall that $u-u_{av}$ on a fixed time slice can be estimated using Poincar\'e inequality implying
that
$$x_n^{-1}W(u-u_{av})(Y,s)\lesssim \left(\fiint_{B_{\delta(Y,s)/2}(Y,s)}|\nabla u|^2\right)^{1/2}=W(\nabla u)(Y,s).
$$
Therefore 
\begin{align}
&\left(\fiint_{Q(X,t)} [D^{1/2}_t((u- u_{av})(1-\varphi))(X',t')]^2\,dt'\,dX'\right)^{1/2}\nonumber\\
&\lesssim
\sum_{i=0}^\infty 2^{-i/2}\fint_{|\tau-t'|\le 2^i x_n^{2}}\tilde{W}(\nabla u)(P,x_n,t')dt'\nonumber\\
&
\le \sum_{i=0}^\infty 2^{-i/2}\fiint_{\{|(P-Q|<x_n/2\}\times\{|s-\tau|<x_n^2/4\}\times\{|\tau-t'|\le 2^i x_n^{2}\}}
\tilde{N}(\nabla u)(Q,s+t')dQ\,ds\,dt'\nonumber,
\end{align}
where we have first shifted the average to be centered above the point $(P,\tau)$ (our vertex). Here
$\tilde{W}$ is an average over slightly enlarged ball (which might be required due to shift from $X$ to $P$) and in the second estimate used the fact that $\tilde{W}(\nabla u)(P,x_n,t')\le \tilde{N}(\nabla u)(Q,s+t')$ for $|(P-Q|<x_n/2\}\times\{|s-\tau|<x_n^2/4\}$.
It follows that the integral in the last line is bounded by $M_t(M(\tilde{N}(\nabla u))(P,\tau)$ and hence
$$\left(\fiint_{Q(X,t)} [D^{1/2}_t((u- u_{av})(1-\varphi))(X',t')]^2\,dt'\,dX'\right)^{1/2}\lesssim 
M_t(M(\tilde{N}(\nabla u))(P,\tau).$$
Hence
\begin{equation}\label{splitw2-}
w_2(X,t)\lesssim M_t(M(\tilde{N}(\nabla u))(P,\tau)+\left(\fiint_{Q(X,t)} [D^{1/2}_t((u_{av}-g)(1-\varphi))(X',t')]^2\,dt'\,dX'\right)^{1/2},
\end{equation}
where for the last term we have that
$$D^{1/2}_t((u_{av}-g)(1-\varphi))(X',t')\lesssim \int_{|t'-\tilde{s}|\ge x_n^{2}}\frac{|u_{av}-g|(\tilde{s})}{|t'-\tilde{s}|^{3/2}}d\tilde{s}.$$
Recall that  $u_{av}$ is defined by \eqref{uav} and by  \eqref{splitwww} we see that
$$D^{1/2}_t((u_{av}-g)(1-\varphi))(X',t')\lesssim \sum_{i=0}^\infty 2^{-i/2} 
\fint_{|t-\tilde{s}|\approx 2^ix_n^2}W(\nabla u)d\tilde{s}.$$
Hence the second term of \eqref{splitw2-} enjoys similar estimates as the first one (using Minkowski) and for both we have the bound:
\begin{equation}\label{splitw2}
w_2(X,t)\lesssim M_t(M(\tilde{N}(\nabla u))(P,\tau).
\end{equation}

Consider now the third term of \eqref{split1}. Mimicking the definition \eqref{defg} we also define
similar averages of the function $f$ by
\begin{equation}\label{defh}
h(s)=\fint_{\pi(Q_r(X,s))}f,
\end{equation}
where $\pi$ is the projection onto the boundary defined earlier.

Observe that $\textstyle\fint_{\Delta} f=h(t)$
and we can write 
\begin{equation}\label{fav}
D^{1/2}_t((g-\textstyle\fint_{\Delta} f)(1-\varphi))=
D^{1/2}_t(h)-D^{1/2}_t((h-\textstyle\fint_{\Delta} f)\varphi)+D^{1/2}_t((g-h)(1-\varphi)),
\end{equation}
where for the first term we may use that $D^{1/2}_tf\in L^p$ which implies by \eqref{defh}  that
$$
(D^{1/2}_th)(s)=\fint_{\pi(Q_r(X,s))}D^{1/2}_tf(z,\tau')\,dz\,d\tau'.
$$
 Evaluating this term for $(X,s)=(x,x_n,s)$ since $(X,t)\in\Gamma(P,\tau)$ we obtain that for some $m>>1$ (depending only on the aperture of the nonntagential region $\Gamma(P,\tau)$ and thus fixed)
we obtain that $Q_r(X,s)\subset Q_{mr}(P,\tau)$ which implies that
\begin{equation}\label{defg3}
|(D^{1/2}_th)(s)|\lesssim \fint_{\pi(Q_{mr}(P,\tau))}|D^{1/2}_tf(z,\tau')|\,dz\,d\tau'\le M(|D^{1/2}_t f|)(P,\tau).
\end{equation}
This estimate takes care of the first term of \eqref{fav}. To handle the second term we recall the following result from \cite{DSa} where the following has been established as Proposition 2.44 there.
\begin{proposition}\label{prop:Product_Rule}
Let $\Omega=\mathcal O\times\R$, where $\mathcal O\subset\R^n$ is a uniform domain
with $n-1$-Ahlfors regular boundary.
Assume that for some $p>1$
 $f\in \dot{L}^p_{1,1/2}(\partial\Omega)$ (with $ \dot{L}^p_{1,1/2}$ defined by \eqref{SpaceL}).
 There exists $H\in L^p(\partial\Omega)$ with $\|H\|_{L^p(\partial\Omega)}\lesssim \|f\|_{\dot{L}^p_{1,1/2}(\partial\Omega)}$ for which the following holds:
 
Consider any parabolic ball $B\subset \R^n\times\R$ of radius $r$ centered at the boundary $\partial\Omega$, let $\Delta=B\cap\partial\Omega$ and  $\varphi$ be a $C_0^\infty(\mathbb R^n\times\R)$ cutoff function such that $\varphi=1$ on $2B$ and $\varphi=0$ outside $3B$ with $\|\partial_t\varphi\|_{L^\infty}\lesssim r^{-2}$ and $\|\nabla_x\varphi\|_{L^\infty}\lesssim r^{-1}$. 
 Then we have the following estimate:
$$\|(f-\textstyle\fint_{\Delta} f)\varphi\|_{\dot{L}^p_{1,1/2}(\partial\Omega)}\lesssim \|H\|_{L^p(32\Delta)}+\|D^{1/2}_t f\|_{L^p(32\Delta)}
\lesssim \|f\|_{\dot{L}^p_{1,1/2}(\partial\Omega)}.$$

Furthermore for all $(x,t)\in 8\Delta$ we have
\begin{equation}\label{245}
|D^{1/2}_t((f-\textstyle\fint_{\Delta} f)\varphi)(x,t)-D^{1/2}_tf(x,t)\varphi(x,t)|\lesssim
H(x,t).
\end{equation}
\end{proposition}\vglue5mm

\noindent We would like to apply \eqref{245} to $h$ not $f$ directly but there is little difference. The function $H$ in the statement of Proposition is constructed from the sharp maximal function of $f$ and thus with $M(H)$ in place of $H$ above the statement does hold for $h$.  Thus we can get for some $r$ with $|r|\lesssim M(H)$:
$$D^{1/2}_t((h-\textstyle\fint_{\Delta} h)\varphi)(y,s)= [(D_t^{1/2}h)\varphi](y,s)+r(y,s).$$
Noticing the average is that of $h$, but the averages $\fint_{\Delta} h$ and $\fint_{\Delta} f$ are only little different
and via the sharp maximal function of $f$ it can be shown that 
$$\left|\fint_{\Delta} h-\fint_{\Delta} f\right|\lesssim r\fint_{3\Delta}f^\sharp\lesssim rM(H)(P,\tau).$$
Therefore (as $D^{1/2}_t\varphi\approx r^{-1}$) we have that
$$|D^{1/2}_t((h-\textstyle\fint_{\Delta} f\displaystyle)\varphi)(y,s)|\lesssim \fint_{\pi(Q_{mr}(P,\tau))}[|D^{1/2}_tf(z,\tau')|+M(H)(z,\tau')]\,dz\,d\tau'
$$
$$+M(H)(P,\tau)\le M(|D^{1/2}_t f|)(P,\tau)+M(M(H))(P,\tau),$$
where as before we have used that $Q_r(y,s)\subset Q_{mr}(P,\tau)$. From this 
$$w_3(P,\tau)\lesssim M(|D^{1/2}_t f|)(P,\tau)+M(M(H))(P,\tau) + N(D^{1/2}_t((g-h)(1-\varphi))).$$

For the last term we have for any $(Y,s)\in Q(X,t)$ that
$$|D^{1/2}_t((u-f)(1-\varphi))|(Y,s)\lesssim \int_{|s-s'|\ge x_n^{2}}\frac{|u-f|(Y,\tilde{s})}{|s-\tilde{s}|^{3/2}}d\tilde{s}$$
$$= \sum_{i=0}^\infty 2^{-i/2}\fint_{|s-\tilde{s}|\approx 2^i x_n^{2}} x_n^{-1}|u-f|(Y,\tilde{s})d\tilde{s}.$$
Recall, that we understand here $f$ as $u\circ\pi$. By using the fundamental theorem of calculus we have
\begin{equation}\label{split33}
|u-f|(Y,\tilde{s})=\int_0^{y_n}|\partial_n u|(y,z_n,\tilde s)dz_n,
\end{equation}
and therefore
$$|D^{1/2}_t((u- f)(1-\varphi))|(Y,s)\lesssim \sum_{i=0}^\infty 2^{-i/2} \fiint_{[0,y_n]\times|s-s'|\approx 2^i x_n^{2}}
|\partial_n u|(y,z_n,s')dz_n\,ds'.$$
We now average this over $Q(X,s)$. As the average of $u$ is $g$ and average of $f$ over $\pi(Q(X,s))$
is $h$ we get that

$$|D^{1/2}_t((g-h)(1-\varphi))|(Y,s)\lesssim \sum_{i=0}^\infty 2^{-i/2} \fiint_{\{|x-y'|<x_n\}\times[0,y_n]\times|s-s'|\approx 2^i x_n^{2}}
|\partial_n u|(y',z_n,s')dy'\,dz_n\,ds'.$$
We now square and average  over $Q(X,s)$. We deal with the sum by split the term $2^{-i/2}$ into two parts to get decay in Cauchy-Schwarz inequality on both terms. That is with the integral being the $b_n$ term in the calculation
$$\left(\sum_n 2^{-i/2}b_n\right)^{2}=\left(\sum_n 2^{-i/4}2^{-i/4}b_n\right)^{2}
\le\textstyle(\sum_n 2^{-i/2})\displaystyle\sum_n 2^{-i/2}|b_n|^{2}.$$
Hence
$$\fiint_{Q(X,t)} |D^{1/2}_t((u- f)(1-\varphi))|^2dY\,ds\lesssim \sum_{i=0}^\infty 2^{-i/2}\left(\fiint_{\{|x-y|<2x_n\}\times[0,3x_n]\times |s-t|\approx 2^ix_n^2}|\partial_n u|\right)^{2}.$$
The integral over the region of integration can be estimated using the non-tangential maximal function of $\partial_n u$. Clearly we may use the $L^1$ version of it  and we get that
$$\fiint_{\{|x-y|<2x_n\}\times[0,3x_n]\times |s-t|\approx 2^ix_n^2}|\partial_n u|dY\,ds\lesssim
\fint_{\{|x-y|<2x_n\}\times  |s-t|\approx 2^ix_n^2} \tilde{N}_{1}(\partial_nu)\,dy\,ds.$$
Recall however, that we have $\tilde{N}_{1}\lesssim \tilde{N}_2$ by H\"older.
Hence the last term is further bounded by $\fint_{\{|x-y|<2x_n\}\times  |s-t|\approx 2^ix_n^2} \tilde{N}_{2}(\partial_nu)\,dy\,ds.$ Considering the maximal functions that can be shifted to the vertex $(P,\tau)$ we can then further estimate it by
$$C M_t(M(\tilde{N}(\nabla u)))(P,\tau),$$
and hence
\begin{equation}\label{splitw3}
w_3(P,\tau)\lesssim M(|D^{1/2}_t f|)(P,\tau)+M(M(H))(P,\tau) +M_t(M(\tilde{N}(\nabla u)))(P,\tau).
\end{equation}

where $(P,\tau)$ is as before the vertex of the non-tangential cone to which the point $(X,t)$ belongs to. \medskip

We have one more term to consider, namely $D^{1/2}_t((\textstyle\fiint_{Q(X,t)}u-\textstyle\fint_{\Delta} f)\varphi)$. Since $D^{1/2}_t(\varphi)\approx r^{-1}$, we just need a good bound on the difference of averages $\textstyle\fiint_{Q(X,t)}u-\textstyle\fint_{\Delta} f$. Adapting the same idea as above, namely \eqref{split33} we obtain aver averaging over $Q(X,t)$:
\begin{equation}\label{zmeko}
\left|\fiint_{Q(X,t)}u-\fint_{\Delta} f\right|\lesssim \int_0^{2r}\fiint_{Q(X,t)}|\partial_n u|\lesssim \int_0^{2r}\fiint_{Q_{mr}(P,\tau)}|\partial_n u|
\end{equation}
$$\le 2r M(N(\nabla u))(P,\tau).
$$
Hence 
$$w_4(X,t)\lesssim M(N(\nabla u))(P,\tau).$$

\medskip
We finish our argument by recalling that $w(X,t)\lesssim w_1(X,t)+w_2(X,t)+w_3(X,t)+w_4(X,t)$ and hence by just finished calculation,  \eqref{splitw1}, \eqref{splitw2}, \eqref{splitw3} and  the estimate above
  we get for $\tilde{N}(D^{1/2}_tu)(P,\tau)=\sup_{(X,t)\in\Gamma(P,\tau)}w(X,t)$:
 $$\tilde{N}(D^{1/2}_tu)(P,\tau)\lesssim M(\tilde{N}(\nabla u))(P,\tau)+M_t(M(\tilde{N}(\nabla u)))(P,\tau)$$
 $$\qquad\qquad\qquad+M(|D^{1/2}_t f|)(P,\tau)+M(M(H))(P,\tau).$$
 Thus by boundedness of Hardy-Littlewood maximal functions $M$ and $M_t$ when $p>1$ we therefore have for some $C>0$:
$$\|\tilde{N}(D^{1/2}_tu)\|_{L^p(\partial\Omega)}\le C[\|\tilde{N}(\nabla u)\|_{L^p(\partial\Omega)}
+\|D^{1/2}_tf\|_{L^p(\partial\Omega)}+\|H\|_{L^p(\partial\Omega)}]$$
$$\quad\lesssim \|\tilde{N}(\nabla u)\|_{L^p(\partial\Omega)}+\|f\|_{\dot{L}^p_{1,1/2}(\partial\Omega)},$$
as desired.
 
 Let us briefly address similar bounds for $\tilde{N}(D^{1/2}_tH_tu)$. Because of $\eqref{splitss}$ and the bounds we have just established above we have that
 $$\left(\fiint_{Q(X,t)}|D^{1/2}_tH_t(v\varphi)|^2\right)^{1/2}\lesssim M(\tilde{N}(\nabla u))(P,\tau).$$
Since
\begin{equation}\label{split1H}
D^{1/2}_tH_tu- D^{1/2}_tH_t(v\varphi) =D^{1/2}_tH_t((u- g)(1-\varphi))+D^{1/2}_tH_t((g-\textstyle\fint_{\Delta}f)(1-\varphi))
\end{equation}
\begin{equation}\qquad\qquad\nonumber
+D^{1/2}_tH_t((\textstyle\fiint_{Q(X,t)}u-\textstyle\fint_{\Delta} f)\varphi).
\end{equation}

we just need to re-analyse the remaining terms for the new operator $D^{1/2}_tH_t$. However, 
$D^{1/2}_tH_t$ and $D^{1/2}_t$ are similar operators, we have
$$D^{1/2}_tH_tf(s)=c\int_{\mathbb R}\frac{f(s)-f(\tau)}{(s-\tau)|s-\tau|^{1/2}}d\tau.$$ 
 It follows that the argument for the first and second terms on the right-hand side of \eqref{split1H} will be identical to 
 the one given for $D_t^{1/2}$ and for the last term we again use the calculation given above without any change.   Hence
 $$\|\tilde{N}(D^{1/2}_tH_tu)\|_{L^p(\partial\Omega)}\le C\|\tilde{N}(\nabla u)\|_{L^p(\partial\Omega)}+\|f\|_{\dot{L}^p_{1,1/2}(\partial\Omega)},$$
as well.\vglue2mm

Now we discuss the key differences between the case $\mathcal O=\mathbb R^n_+$ and the case 
$\mathcal O$ is an arbitrary uniform domain with $n-1$-Ahlfors regular boundary. Some steps are nearly identical, for example after replacing $x_n$ by $h=\delta(X,t)=\delta(X)$ we see that
the estimate for \eqref{splitw1} holds as this is an interior estimate. Same is true for \eqref{splitw2}.
Observe also that Proposition \ref{prop:Product_Rule} does not require any modification (we used it for one of the $w_3$ terms).

Steps that do need a serious rethink are those where we have integrated along the straight line from an interior point to the boundary (c.f.  \eqref{split33}).  These estimates rely on the projection $\pi:\Omega\to \partial\Omega$ which only really make sense on graph-like domains. \vglue1mm

Here we modify the idea of the paper \cite{DPext}. The whole integration here takes place on a single time slice of the domain $\Omega$ and hence let us drop for now the variable $t$ completely.

Consider a function $u:\mathcal O\to\mathbb R$. Let us define space-only averages of $u$, that is
$$u_{av}(X)=\fiint_{|X-Z|<\delta(Y)/8} u(Z)\,dZ.$$

For a fixed point $X\in \mathcal O$ and let $h=\delta(X)$. Using the fundamental theorem of calculus we clearly see that this average is related to the average defined in \eqref{uav}
and we will have for $u$ on the time slice $t_0$ (slightly abusing the notation):
\begin{equation}\label{zzb}
|u_{av}(X)-u_{av}(t_0)|\lesssim  h\fiint_{B(X,\delta(X)/4)}|\nabla u|.
\end{equation}

Recall that  $\tilde{S}(X)=\{q\in\partial\Omega: X\in\tilde\gamma(q)\}$. By Proposition 3.3 of \cite{DPext} we have $\sigma(\tilde{S}(X))\approx h^{n-1}$, where $\sigma=\mathcal H^{n-1}\Big|_{\partial\mathcal O}$ is the natural $n-1$-dimensional Hausdorff measure on $\partial\mathcal O$.

For any $q\in \tilde{S}(Y)$ we estimate the difference $|u_{av}(X)-f(q)|$. 
The argument uses both key properties of our domain $\mathcal O$, namely the existence of corkscrew points and interior Harnack chain condition.  We claim that

\begin{equation}\label{Conebound}
|u_{av}(X)-f(q)|  \lesssim A_{\tilde{a}}(\nabla u)(q)
\end{equation}
where 
$$A_{\tilde{a}}(\nabla u)(q) =  \iint_{\widetilde\gamma_{\tilde{a}}^{2h}(q)} |\nabla v|(Z) \delta(Z)^{1-n} dZ.$$
Here, the parameter $\tilde{a}$ (the aperture of the corkscrew region $\widetilde\gamma(q)$) will be determined later and the superscript $2h$ mean that we will truncate the corkscrew region at the height $2h$, that is
$\widetilde\gamma^{2d}(q):= \widetilde\gamma(q) \cap B(q,2h)$.

\medskip

Proof of \eqref{Conebound}:  Since $X \in \widetilde\gamma_a(q)$, it follows that $X \in \gamma_{2a}(q)$ and so there is a sequence of corkscrew points $X_j$ associated to
 the point $q$ at scales 
$r_j \approx 2^{-j}h$, $j=0,1,2,\dots$ with $X_0=X$. By the Harnack chain condition, for each $j$ there is a number $N$ and a constant $C>0$ such that
there exists  $n(j) \leq N$ balls $B^{(j)}_k$ of radius $\approx 2^{-j}h$ with $C B^{(j)}_k \subset \mathcal O$, $X_{j-1} \in B^{(j)}_1$, 
$X_j \in B^{(j)}_n$, and $B^{(j)}_k \cap  B^{(j)}_{k+1} \ne \emptyset$. 
Therefore we can find another chain of balls with the same properties for a larger but fixed choice of $N$ so that
$4 B^{(j)}_k \subset \mathcal O$ 
and $B^{(j)}_{k+1} \subset 2B^{(j)}_{k}$.

Considering the whole collection of balls $B_{k}^{(j)}$ for all $j=0,1,2,\dots$ and $k=1,\dots, n(j)\le N$ it follows that we have an infinite chain of balls, the first of which contains $X=X_0$, converging to the boundary point $q$, with the property that any pair of consecutive balls in
the chain have roughly the same radius and their $4$-fold enlargements are contained in $\mathcal O$. We relabel these
balls $B_j(X_j,r_j)$ with centers $X_j$ and radii $r_j \approx t^{-j}h$ for some $t < 1$ depending on $N$. 

We next claim that, for any $j=0,1,2,\dots$, 
\begin{equation}\label{Averages}
\left|\fiint_{B_j} u(Z) dZ - \fiint_{B_{j+1}} u(Z) dz\right| \lesssim t^{-j}\fiint_{2B_j}|\nabla u|(Z)dZ\approx \iint_{2B_j}|\nabla u|(Z)\delta(Z)^{1-n}dZ.
\end{equation}

To prove \eqref{Averages}, define the map $T(Y) =  r_j/r_{j+1}(Y-X_j) + X_{j+1}$ from $B_j$ to $B_{j+1}$. Then
\begin{equation}
|u(T(Y)) - u(Y)| \leq \int_{\ell \in [Y, T(Y)]} |\nabla u(\ell)|\,\, d\ell
\end{equation}
where $ [Y, T(Y)]$ is the line segment from $Y$ to $T(Y)$.
Averaging $Y$ over $B_j$, using the triangle inequality, and observing that the collection of lines $[Y,T(Y)]$ is contained  in $2B_j$, gives us what we want.

The claim \eqref{Conebound} results from summing the averages in \eqref{Averages} as follows.

Set $$U_j := \fiint_{B_j} u(Z)dZ - \fiint_{B_{j+1}} u(Z)dZ$$
Because $u$ attains value $f(q)$ on the boundary, the averages $\fiint_{B_j} u(Z)dZ $ are converging to $f(q)$
for $\sigma$-a.e. $q\in\partial\Omega$. It follows that
$$|u_{av}(X)-f(q)|\le \sum_{j=0}^\infty |U_j|\lesssim \sum_{j=0}^\infty \iint_{2B_j}|\nabla u|(Z)\delta(Z)^{1-n}dZ.$$
We have that each $X_j\in \gamma_{1+2a}(q)$ and hence for some $\tilde{a}>>2a$ (independent of $q$) we will have that the enlarged ball $2B_j\subset \gamma_{\tilde a}(q)$. Hence \eqref{Conebound} follows.

Since \eqref{Conebound} holds for all $q\in \tilde{S}(X)$ we integrate the inequality over this set. This gives us
$$
\left|\int_{\tilde{S}(X)} [u_{av}(X)-f(q)]d\sigma(q)\right|\lesssim \int_{\tilde{S}(X)}A_{\tilde{a}}d\sigma
=\int_{\tilde{S}(X)}\iint_{\widetilde\gamma_{\tilde{a}}^{2h}(q)} |\nabla v|(Z) \delta(Z)^{1-n} dZ\,d\sigma(q).
$$
Let $T(X)=\bigcup_{q\in \tilde{S}(X)} \gamma_{\tilde{a}}^{2h}(q)$. Since for any $Z\in T(X)$ we have that
$\sigma((\tilde{S}(Z))\lesssim \delta(Z)^{n-1}$

By exchanging the order in the last integration get get that
$$\left|\int_{\tilde{S}(X)} [u_{av}(X)-f(q)]d\sigma(q)\right|\lesssim \iint_{T(X)} |\nabla v|(Z)  dZ.$$
It follows that
\begin{equation}\label{zza}
\left| u_{av}(X)-f_{av}\right|\lesssim h^{1-n}\iint_{T(X)} |\nabla v|(Z)  dZ,
\end{equation}
where  $f_{av}=\sigma(\tilde{S}(X))^{-1}\int_{\tilde{S}(X)} f\,d\sigma$. 
We now use this on the time-slices $Q(X,t)\cap\{\tau=const\}$ for $\tau\in (t-h^2/16,t+h^2/16)$, integrate in $\tau$ and average. Let $\Delta$ be the set $\tilde{S}(X)\times (t-h^2/16,t+h^2/16)$ any $\mu$ the product measure $d\mu=d\sigma\, dt$. It follows that
$$\left|\fiint_{Q(X,t)} u-\mu(\Delta)^{-1}\int_{\Delta} f \right|\lesssim \fint_{t-h^2/16}^{t+h^2/16}[|u_{av}(X,\tau)-f_{av}(\tau)|+|u_{av}(\tau)- u_{av}(X,\tau) |]\,d\tau
$$
and after using \eqref{zza} for the first term and \eqref{zzb} for the second term we get that
$$\left|\fiint_{Q(X,t)} u-\mu(\Delta)^{-1}\int_{\Delta} f \right|\lesssim  h^{-n-1}\iint_{T(X)\times (t-h^2/16,t+h^2/16)} |\nabla v|  d Z\,d \tau.$$
This is a replacement of the estimate needed for \eqref{zmeko}. 

An analogous calculation also applies to the term containing $D^{1/2}_t((g-h)(1-\varphi))$. Thus
\eqref{zza} is the key to modifying  the proof to this more general geometric setting.
Hence the claim follows.
\end{proof}


\bibliographystyle{alpha}

\end{document}